\documentclass[1p,final]{elsarticle}
\usepackage{amsfonts,color,morefloats,pslatex}
\usepackage{amssymb,amsthm, amsmath,latexsym}

\newtheorem{theorem}{Theorem}
\newtheorem{lemma}[theorem]{Lemma}

\newtheorem{corollary}[theorem]{Corollary}

\newtheorem{conj}{Conjecture}

\newcommand{\ord}{{\mathrm{ord}}}

\newcommand{\tr}{{\mathrm{Tr}}}
\newcommand{\Tr}{{\mathrm{Tr}}}

\newcommand{\gf}{{\mathrm{GF}}}

\newcommand{\PSL}{{\mathrm{PSL}}}
\newcommand{\PGL}{{\mathrm{PGL}}}
\newcommand{\GA}{{\mathrm{GA}}}

\newcommand{\AGL}{{\mathrm{AGL}}}

\newcommand{\cP}{{\mathcal{P}}} 
\newcommand{\cB}{{\mathcal{B}}} 
\newcommand{\cL}{{\mathcal{L}}} 
 
\newcommand{\C}{{\mathcal{C}}}

\newcommand{\cQ}{{\mathcal{Q}}}
\newcommand{\cN}{{\mathcal{N}}}
\newcommand{\cJ}{{\mathcal{J}}}

\newcommand{\AG}{{\mathrm{AG}}}

\newcommand{\bD}{{\mathbb{D}}}

\begin{document}

\begin{frontmatter}



\title{Infinite families of $2$-designs from $\GA_1(q)$ actions   
}

%\date{}
\author{Hao~Liu and Cunsheng Ding}
\ead{hliuar@ust.hk, cding@ust.hk} 
\address{Department of Computer Science and Engineering, The Hong Kong University of Science and Technology, Clear Water Bay, Kowloon,
Hong Kong}

\begin{abstract} 
Group action is a standard approach to obtain $t$-designs. In this approach, selecting 
a specific permutation group with a certain degree of transitivity or homogeneity and a proper set 
of base blocks is important for obtaining $t$-$(v, k, \lambda)$ designs with computable 
parameters $t, v, k$, and $\lambda$. The general affine group $\GA_1(q)$ is $2$-transitive on $\gf(q)$, 
and has relatively a small size. In this paper, we determine the parameters of a number of 
infinite families of $2$-designs obtained from the action of the group $\GA_1(q)$ on certain 
base blocks, and demonstrate that some of the $2$-designs give rise to linear codes with optimal or 
best parameters known. Open problems are also presented.   
\end{abstract}

\begin{keyword}
General affine group \sep linear code \sep $t$-design.

\MSC  05B05 \sep 51E10 \sep 94B15 

\end{keyword}

\end{frontmatter}

\section{Introduction}

Let $\cP$ be a set of $v \ge 1$ elements, and let $\cB$ be a set of $k$-subsets of $\cP$, where $k$ is
a positive integer with $1 \leq k \leq v$. Let $t$ be a positive integer with $t \leq k$. The pair
$\bD = (\cP, \cB)$ is called a $t$-$(v, k, \lambda)$ {\em design\index{design}}, or simply {\em $t$-design\index{$t$-design}}, if every $t$-subset of $\cP$ is contained in exactly $\lambda$ elements of
$\cB$. The elements of $\cP$ are called points, and those of $\cB$ are referred to as blocks.
We usually use $b$ to denote the number of blocks in $\cB$.  A $t$-design is called {\em simple\index{simple}} if $\cB$ does not contain repeated blocks. In this paper, we consider only simple 
$t$-designs.  A $t$-design is called {\em symmetric\index{symmetric design}} if $v = b$. It is clear that $t$-designs with $k = t$ or $k = v$ always exist. Such $t$-designs are {\em trivial}. In this paper, we consider only $t$-designs with $v > k > t$.
A $t$-$(v,k,\lambda)$ design is referred to as a {\em Steiner system\index{Steiner system}} if $t \geq 2$ and $\lambda=1$, and is denoted by $S(t, k, v)$.  

Let $\cP$ be a set of $v \ge 1$ elements, and let $G$ be a permutation group on $\cP$. 
$G$ is said to be transitive on $\cP$, if for any two elements $x$ and $y$ in $\cP$ 
there is a $\pi \in G$ such that $\pi(x)=y$. $G$ is said to be $t$-transitive on $\cP$, 
if for any two ordered $t$-subsets of $\cP$, there is a $\pi \in G$ such that $\pi$ sends 
the former to the latter. $G$ is said to be $t$-homogeneous on $\cP$, 
if for any two $t$-subsets of $\cP$, there is a $\pi \in G$ such that $\pi$ sends 
the former to the latter. If $G$ is $t$-transitive on $\cP$, it must be $t$-homogeneous on $\cP$. 
But the converse may not be true.    
   
A classical method of constructing $t$-designs by group action is described in the following 
theorem \cite[p. 175]{BJL}. 

\begin{theorem}\label{thm-designbygroupaction}
Let $\cP$ be a set of $v \ge 1$ elements, and let $G$ be a permutation group on $\cP$. Let 
$B \subset \cP$ be a subset with at least two elements. Define 
$$ 
G(B)=\{g(B): g \in G\}, 
$$    
where $g(B)=\{g(b): b \in B\}$. If $G$ is $t$-homogeneous on $\cP$ and $|B| \geq t$, then 
$(\cP, G(B))$ is a $t$-$(v, k, \lambda)$ design with 
$$ 
k=|B|, \ \lambda=b\frac{\binom{k}{t}}{\binom{v}{t}}=
                 \frac{|G|}{|G_B|} \frac{\binom{k}{t}}{\binom{v}{t}}, 
$$
where $b= |G|/|G_B|$ and $G_B=\{g \in G: g(B)=B\}$ is the setwise stabiliser of $B$. 
\end{theorem} 

To apply Theorem \ref{thm-designbygroupaction}, one has to design or select a point set 
$\cP$ and a permutation group $G$ on $\cP$, and choose a base block $B \subset \cP$ properly, 
so that it is possible to determine $|G_B|$ and thus the parameter $\lambda$ of the design.   

Let $q$ be a prime power. The general affine group $\GA_1(q)$ of degree one consists of all the following 
permutations of the set $\gf(q)$: 
$$ 
\pi_{(a,b)}(x)=ax+b,  
$$ 
where $a \in \gf(q)^*$ and $b \in \gf(q)$. It is a group under the function composition operation, 
and is interesting, as it is doubly transitive on $\gf(q)$ and has a small group size. This group 
is also denoted by $\AGL(1,q)$ in many references, and can be written as 
$$ 
\GA_1(q)\sim\gf(q) \rtimes \gf(q)^*, 
$$ 
which is the external semidirect product of the additive group of $\gf(q)$ and the multiplicative 
group of $\gf(q)$. 

In this paper, we will employ the group $\GA_1(q)$ and Theorem \ref{thm-designbygroupaction} 
to construct a number of infinite families of $2$-designs and determine their parameters. We 
will also demonstrate that some of the designs presented in this paper yield linear codes with 
optimal or best known parameters. 

\section{The general construction of $t$-designs from the action of $\GA_1(q)$} 

As a corollary of Theorem \ref{thm-designbygroupaction}, we have the following. 

\begin{corollary}\label{cor-ourmaincorr} 
Let $q$ be a prime power. Let $\cP = \gf(q)$ and $\cB = \GA_1(q)(B)$, where $B$ is any $k$-subset of 
$\gf(q)$ with $k \geq 2$. Then $(\cP, \cB)$ is a $2$-$(q, k, \lambda)$ 
design, where 
$$
\lambda=\frac{|\GA_1(q)|}{|G_B|} \frac{\binom{k}{2}}{\binom{q}{2}} = \frac{k(k-1)}{|G_B|}. 
$$ 
\end{corollary} 

To obtain $2$-designs with computable parameters from Corollary \ref{cor-ourmaincorr}, 
one has to choose the subset $B$ properly. In general, computing the parameters of the designs in 
Corollary \ref{cor-ourmaincorr} is a very hard task, since the determination of $|G_B|$ would be difficult in most cases.

\section{The first family of $2$-designs} 

In this section, we employ the subgroups of the multiplicative group of $\gf(q)$ as base blocks. Let $q$ be a prime power with $q-1 = ef$, where $e$ and $f$ are positive integers, and let $\gamma$ be a primitive element of $\gf(q)$. The cyclotomic classes of order $e$  are defined by 
$$\C^e_i :=\{\gamma^{ke+i},~0\le k\le f-1\}, \ 0 \leq i \leq e-1,$$
which are the cosets of the subgroup $\C_0^e$ in $\gf(q)^*$. The cyclotomic numbers of order $e$ are defined as $$(s,t)_e := |(\C^e_s+1)\cap \C^e_t|,$$ 
where $0 \leq s \leq e-1$ and $0 \leq t \leq e-1$.

\begin{lemma}\label{lemmaCe0stab}
Let $q$ be a prime power with $q-1 = ef$, where $f \geq 2$, and let $\gamma$ be a primitive element of $\gf(q)$. Let $\mathcal{C}^e_0 = \{\gamma^{ke},~k=0,1,...,f-1\}$ be the subgroup of order $f$ in $\gf(q)^*$. Then the stabiliser of $\mathcal{C}^e_0$ and $\{0\}\cup \mathcal{C}^e_0$ in $\GA_1(q)$ is the cyclic group 
$$C_f:=\{\pi(x) = \gamma^{ek}x,~0\le k\le f-1\}$$
 of order $f$.  
\end{lemma}

\begin{proof}
First we consider the set $C^e_0$.
Assume $\pi(x) = ax+b\in\GA_1(q)$ fixes $\C^e_0$. Clearly $a\neq 0$, as $|\C^e_0|=f \geq 2$. If $b=0$, it is straightforward to verify that $a \in \C^e_0$ and $\pi\in C_f$.

If $b\neq 0$, we have $a\C^e_0+b = \C^e_0$. Note that the elements in $\C^e_0$ are all the roots of $x^f=1$. 
We have then $\sum_{x\in\C^e_0} x = 0$. Summing up the elements on the two sides of $a\C^e_0+b = \C^e_0$, 
we have $bf =0$. By definition, $\gcd(q, f)=1$. It then follows from $fb=0$ that $b=0$,  which is contrary to 
our assumption that $b\neq 0$.

The desired conclusion for the base block $C^e_0\cup\{0\}$ is similarly proved. The details of proof are omitted. 
\end{proof}

Combining Lemma \ref{lemmaCe0stab} and Corollary \ref{cor-ourmaincorr}, we obtain the following. 

\begin{theorem}\label{thm-mainr1} 
Let $q$ be a prime power with $q-1 = ef$, where $f \geq 2$. Define $\cP=\gf(q)$. 
Let $\cB:=\GA_1(q)(B)$ and $\widehat{\cB} = \GA_1(q)(\widehat{B})$, where $B:=\C^e_0$ is the set of all $e$-th powers in $\gf(q)^*$ and $\widehat{B} = B\cup\{0\}$. 
Then $(\cP, \cB)$ is a $2$-$(q, (q-1)/e, (q-1-e)/e)$ design and  $(\cP, \widehat{\cB})$ is a $2$-$(q, (q-1+e)/e, (q-1)/e)$ design. 
\end{theorem}

\section{The second family of $2$-designs} 

In this section, we consider $2$-designs under the action of $\GA_1(q)$, where the base block $B$ is defined by the set $\cQ$ of all nonzero squares in $\gf(q)$ in some way. To determine the parameters of these $2$-designs, 
we need cyclotomic numbers of order 2, which are documented in the following lemma \cite{Storer67}.

\begin{lemma}\label{lemmaCycNum2}
Let $q$ be a power of an odd prime, and let $\cQ$ (respectively, $\cN$) be the set of nonzero squares (respectively, nonsquares). 
Then 
\begin{align*}
|(\cQ+1)\cap \cQ| =  (0,0)_2 = \begin{cases}
\frac{q-3}{4} & \mbox{ if } q \equiv 3 \pmod{4}, \\
\frac{q-5}{4} & \mbox{ if } q \equiv 1 \pmod{4} 
\end{cases}
\end{align*}
and 
\begin{align*}
|(\cQ+1)\cap \cN| =  (0,1)_2 = \begin{cases}
\frac{q+1}{4} & \mbox{ if } q \equiv 3 \pmod{4}, \\
\frac{q-1}{4} & \mbox{ if } q \equiv 1 \pmod{4} 
\end{cases}
\end{align*}
and 
\begin{align*}
|(\cN+1)\cap \cN| =  (1,1)_2 = |(\cN+1)\cap \cQ| =  (1,0)_2 = \begin{cases}
\frac{q-3}{4} & \mbox{ if } q \equiv 3 \pmod{4}, \\
\frac{q-1}{4} & \mbox{ if } q \equiv 1 \pmod{4}. 
\end{cases}
\end{align*}
\end{lemma}

We first consider the $2$-design with the base block $B:= \cQ\cap(\cQ+\tau)$. The cardinality of $B$, i.e., 
the block size $k$ of our design, is given by the cyclotomic number of order 2 described in Lemma \ref{lemmaCycNum2}. 
To determine the parameter $\lambda$ of the $2$-design, we will determine the stabiliser of $B$ in  
$\GA_1(q)$ below. 

Let $\eta$ denote the quadratic character of $\gf(q)$ defined by  
\begin{align*}
\eta(x) = \begin{cases}
1, &~x \in \cQ, \\
0, &~x = 0, \\
-1,&~x\in\cN. 
\end{cases}
\end{align*}
For simplicity, we use the following symbols for certain sets and numbers.
\begin{align*}
\eta'_a &= \frac{1-\eta(a)}{2}, \\
\Psi_{0,a,b}& = \sum_{x\in\gf(q)}\eta(x)\eta(x-a)\eta(x-b), \\
\cJ_{a,b} &= (\cQ+a)\cap(\cQ+b), \\
\cJ_{0,a,b} &= \cQ\cap(\cQ+a)\cap(\cQ+b), \\
%\cK_{a,b} &= (\cQ+a)\setminus (\cQ+b), \\
\cL_1&= \{x\in\gf(q),~ \eta(x)\eta(x-a)\eta(x-b) = 1\},
\end{align*}
where $a,b\in\gf(q)^*$, $a\neq b$.

We need the following lemma on the cardinality of $\cJ_{0,a,b}$.
\begin{lemma}\label{lemmaJ0ab}
Let symbols be the same as above. We have
\begin{align*}
|\cJ_{0,a,b}|  = \begin{cases}
\displaystyle\frac{q-6+\Delta_{a,b}+\Psi_{0,a,b}}{8} ~&\mbox{ if } q\equiv3\pmod{4}, \\
\displaystyle\frac{q+\Delta_{a,b}-N_3+\Psi_{0,a,b}}{8} ~&\mbox{ if } q\equiv 1\pmod{4},\\
\end{cases}
\end{align*}
where 
$$N_3 = |\{a,b,a-b\}\cap\cQ| = 3-\eta'_a-\eta'_b-\eta'_{a-b}=(3+\eta(a)+\eta(b)+\eta(a-b))/2$$ 
and 
$$\Delta_{a,b} = \eta(a)\eta(b)+\eta(a-b)\eta(a)+\eta(b-a)\eta(b).$$ 
\end{lemma}
\begin{proof}
By the Inclusion-Exclusion Principle, we have 
\begin{equation}\label{eq2}
|\cQ|+|\cQ+a|+|\cQ+b| - 2(|\cJ_{0,a}|+|\cJ_{a,b}|+|\cJ_{0,b}|) = |\cL_1|+|\Delta_3|-4|\cJ_{0,a,b}|,
\end{equation}
where $\Delta_3 := \{0,a,b\}\cap\left(\cQ\triangle(\cQ+a)\triangle(\cQ+b)\right)$, and here $\triangle$ stands for the symmetric difference operator.

One can easily verify the following: 
$$|\cQ|=|\cQ+a|=|\cQ+b|=\frac{q-1}{2},$$
$$|J_{a,b}| = |Q \cap (Q+a-b)|= \begin{cases}
\frac{q-3}{4} & \mbox{ if } q \equiv 3 \pmod{4}, \\
\frac{q-5}{4} & \mbox{ if } q \equiv 1 \pmod{4} \text{ and } a-b\in\cQ, \\
\frac{q-1}{4} & \mbox{ if } q \equiv 1 \pmod{4} \text{ and } a-b\in\cN, \\
\end{cases}$$ 
and 
$$|\cL_1| = \frac{q-3+\Psi_{0,a,b}}{2}.$$
Plugging them into \eqref{eq2}, we obtain  
\begin{align}\label{eqn-ding16}
|\cJ_{0,a,b}|  = \begin{cases}
\displaystyle\frac{q-9+2|\Delta_3|+\Psi_{0,a,b}}{8} ~&\mbox{ if } q\equiv3\pmod{4}, \\
\displaystyle\frac{q-3+2|\Delta_3|-N_3+\Psi_{0,a,b}}{8} ~&\mbox{ if } q\equiv 1\pmod{4},\\
\end{cases}
\end{align}
where $N_3 = |\{a,b,a-b\}\cap\cQ| = 3-\eta'_a-\eta'_b-\eta'_{a-b}=(3+\eta(a)+\eta(b)+\eta(a-b))/2$.

Furthermore, we have
\begin{align*}
|\Delta_3|& = \frac{1-\eta(a)\eta(b)}{2} +\frac{1-\eta(a)\eta(a-b)}{2} +\frac{1-\eta(b-a)\eta(b)}{2} \\
&=\frac{3-\eta(a)\eta(b)-\eta(a-b)\eta(a)-\eta(b-a)\eta(b)}{2}.
\end{align*}
One obtains the desired conclusions after plugging these expressions into \eqref{eqn-ding16}.
\end{proof}

To determine the stabiliser of our base block $B=\cQ\cap(\cQ+\tau)$, we need the following bound on a type of character sums \cite[Theorem 5.41]{LN97}.

\begin{theorem}\label{thmWeilBound} 
Let $\psi$ be a multiplicative character of order $m$ on $\gf(q)$ and let $f\in \gf(q)[x]$ be a monic polynomial of positive degree that is not an $m$-th power of a polynomial. Let $d$ be the number of distinct roots of $f$ in its splitting field over $\gf(q)$. Then for every $ a\in\gf(q)$ we have 
$$\left\vert\sum_{c\in\gf(q)}\psi(af(c))\right\vert \le (d-1)q^{1/2}.$$
\end{theorem}

\begin{lemma}\label{lemmaStabQcapQtau}
Let $q>9$ be a power of an odd prime, and $\cQ$ (respectively, $\cN$) be the set of nonzero squares (respectively, nonsquares) in $\gf(q)$. For any $\tau \in \gf(q)^*$, the stabliser of $\cQ\cap(\cQ+\tau)$ in $\GA_1(q)$ is 
\begin{align*}
G_{\cQ\cap(\cQ+\tau)} = \begin{cases}
\{\mathbf{1}\} ~&\mbox{ if } q \equiv 3 \pmod{4}, \\
\{\mathbf{1},\pi(x) = \tau-x\} ~&\mbox{ if }  q\equiv 1 \pmod{4}.\\
\end{cases}
\end{align*}
\end{lemma}

\begin{proof}
Assume $\pi(x) = ax+b$ fixes $B:=\cQ\cap(\cQ+\tau)$. Clearly $a\neq 0$ and this gives
\begin{equation}\label{eq1}
(a\cQ+b)\cap(a\cQ+a\tau + b)  = \cQ\cap(\cQ+\tau).
\end{equation}

If $a\in\cQ$, then (\ref{eq1}) becomes
$$(\cQ + b)\cap(\cQ+a\tau + b)  = \cQ\cap(\cQ+\tau).$$

If $\{b,a\tau+b\} = \{0,\tau\}$, the equation naturally holds. In this case, we have $\pi(x)=x$ when $q\equiv 3\pmod{4}$, and $\pi(x) = x$ or $\tau-x$ when $q\equiv 1\pmod{4}$.

If $\{b,a\tau+b\} \neq \{0,\tau\}$, Without loss of generality, we can assume $0\notin \{b,a\tau+b\} $ (otherwise we can deduct $\tau$ from both sides of Equation \eqref{eq1}).
This leads to 
$$(\cQ + b)\cap(\cQ+a\tau + b)\cap\cQ  = \cQ\cap(\cQ+\tau). $$
Following the symbols in previous discussions, we have
$$\cJ_{0,b,a\tau+b} = \cJ_{0,\tau}.$$
However, by Lemmas \ref{lemmaCycNum2} and \ref{lemmaJ0ab}, when $q\equiv 3\pmod{4}$ we have
\begin{align*}
| \cJ_{0,\tau} |-|  \cJ_{0,b,a\tau+b}| &= (q-3)/4 - (q-6+\Psi_{0,b,a\tau+b}+\Delta_{a,b})/8 
= (q-\Psi_{0,b,a\tau+b}-\Delta_{a,b})/8
\end{align*}
where $\Delta_{a,b} = \eta(a)\eta(b)+\eta(a-b)\eta(a)+\eta(b-a)\eta(b)$.

Clearly, we have $|\Delta_{a,b}|\le 3$. By Theorem \ref{thmWeilBound}, we have $|\Psi_{0,b,a\tau+b}|\le 2\sqrt{q}$. Thus when $q>9$, we have $| \cJ_{0,\tau} |-|  \cJ_{0,b,a\tau+b}| > 0$, which is a contradiction. 
 
When $q\equiv 1\pmod{4}$, we can reach the same conclusion since $|N_3|\le 3$. Thus the assumption of $\{b,a\tau+b\} \neq \{0,\tau\}$ does not hold in the case of $a \in \cQ$.
 
 If $a\in\cN$, then \eqref{eq1} becomes
 $$(\cN+b)\cap(\cN+\tau a+b) = \cQ\cap(\cQ+\tau).$$
 Since $(\cN+\tau)\cap(\cQ+\tau) = \emptyset$, we have $b\neq \tau$. Intersecting $\cN+b$ with both sides of 
 the equation above, we obtain  
  $$(\cN+b)\cap(\cN+\tau a+b) = \cQ\cap(\cQ+\tau)\cap(\cN+b) = \left(\cJ_{0,\tau}\setminus\{b\}\right) \setminus \cJ_{0,\tau,b}.$$
By a similar argument on the cardinality of both sides as in the case of $a\in\cQ$, we will arrive at the  contradiction that the set on the left side has a larger cardinality when $q>9$.
\end{proof}

\begin{theorem}\label{thm-mainr2} 
Let $q$ be a power of an odd prime, and let $\cQ$ denote the set of all nonzero squares in $\gf(q)$. 
Define $\cP=\gf(q)$ and 
\begin{eqnarray*}
B_\tau=\cQ \cap (\cQ+\tau)
\end{eqnarray*} 
where $\cQ+\tau =\{x+\tau: x \in \cQ\}$ and $\tau \in \gf(q)^*$. 
Then $(\cP, \GA_1(q)(B_\tau))$ is a $2$-$(q, k, \lambda)$ design, where 
\begin{eqnarray*}
k=\left\{ 
\begin{array}{ll}
\frac{q-3}{4} & \mbox{ if } q \equiv 3 \pmod{4}, \\
\frac{q-5}{4} & \mbox{ if } q \equiv 1 \pmod{4} \mbox{ and } \tau \in \cQ, \\
\frac{q-1}{4} & \mbox{ if } q \equiv 1 \pmod{4} \mbox{ and } \tau \in \cN  \\
\end{array} 
\right. 
\end{eqnarray*} 
and 
\begin{eqnarray*}
\lambda=\left\{ 
\begin{array}{ll}
k(k-1) & \mbox{ if } q \equiv 3 \pmod{4}, \\
\frac{k(k-1)}{2} & \mbox{ if } q \equiv 1 \pmod{4}. 
\end{array} 
\right. 
\end{eqnarray*}  
\end{theorem} 

\begin{proof} 
When $q \leq 9$, the conclusions can be verified by hands. Now we assume that $q>9$. 
The conclusion on the block size directly follows from Lemma \ref{lemmaCycNum2}, while the conclusion on $\lambda$ follows from Corollary \ref{cor-ourmaincorr} and Lemma \ref{lemmaStabQcapQtau}.
\end{proof}

Let $q$ be a prime power. The projective general linear group $\PGL_2(q)$ consists of all the following 
permutations of the set $\{\infty\} \cup \gf(q)$: 
$$ 
\pi_{(a,b,c,d)}(x)=\frac{ax+b}{cx+d} 
$$ 
with $ad-bc \neq 0$, and the following conventions: 
\begin{itemize}
\item $\frac{a}{0} = \infty $ for all $a \in \gf(q)^*$. 
\item $\frac{\infty a +b}{\infty c +d} =\frac{a}{c}$. 
\end{itemize} 
Each $\pi_{(a,b,c,d)}$ is a permutation on the set $\gf(q) \cup \{\infty\}$. 
$\PGL_2(q)$ is a group under the function composition operation.

Let $q$ be a prime power. The projective special linear group $\PSL_2(q)$ consists of all the following 
permutations of the set $\{\infty\} \cup \gf(q)$: 
$$ 
\pi_{(a,b,c,d)}(x)=\frac{ax+b}{cx+d} 
$$ 
with $ad-bc =1$. 
$\PSL_2(q)$ is a group under the function composition operation, and is a subgroup of $\PGL_2(q)$.

Next we determine the parameters of the $2$-design obtained under the action of $\GA_1(q)$ on the base block 
$B_\tau:=\cQ\triangle(\cQ+\tau)$, where $\tau \in \gf(q)^*$. Its block size can be obtained from Lemma \ref{lemmaCycNum2}. To determine the parameter $\lambda$, we need  to know the stabilisers of $\cQ$ and $\cQ\cup\{0\}$ in the group $\PGL_2(q)$. The following lemma is well known and easy to prove. 

\begin{lemma}\label{lemmaStabTrans}
Let $\widetilde{G} := \PGL_2(q)$ and $\sigma\in\PGL_2(q)$ be a permutation on $\gf(q)\cup\{\infty\}$. For a subset $B\subset\gf(q)\cup\{\infty\}$ and its stabiliser $\widetilde{G}_{B}$ in $\PGL_2(q)$, the stabiliser of $B^\sigma := \{ \sigma(x),~x\in B\}$ is $\widetilde{G}_{B^\sigma} = \sigma\widetilde{G}_B\sigma^{-1}$.

\end{lemma}

\begin{theorem}\label{thmStabQ}
Let $q>9$ be an odd prime power and $\widetilde{G}=\PGL_2(q)$ be the projective general linear group acting on the projective line $\gf(q)\cup\{\infty\}$. 
Let $\cQ$ be the set of nonzero squares in $\gf(q)$. Put $V_0 := \cQ\cup\{0\}$ and $U_0 = \cQ\cup\{\infty\}$.

1) The stabiliser of $\cQ$ in $\PGL_2(q)$ is 
\begin{align*}
\widetilde{G}_{\cQ} = 
\{\pi_{a,0,0,1}(x) = ax, ~a\in\cQ\}\cup\{\pi_{0,b,1,0}(x) = b/x,~b\in\cQ\}, &\mbox{ where } ~q\equiv 1\pmod{4}. 
\end{align*}

2) The stabiliser of $V_0 : =\cQ\cup\{0\}$ and $U_0 = \cQ\cup\{\infty\}$ in $\PGL_2(q)$ is 
\begin{align*}
\widetilde{G}_{V_0} = \widetilde{G}_{U_0}= \{\pi_{a,0,0,1}(x) = ax, \ ~a\in\cQ\}.
\end{align*}

\end{theorem}

\begin{proof}
1) Assume that $\pi(x)  = \frac{ax+b}{cx+d}$ is an element of the stabiliser of $\cQ$ in $\PGL_q(q)$, where $ad\neq bc$. Define another polynomial function on $\gf(q)$ as 
$$f(x):=(ax^2+b)(cx^2+d).$$
Since $\pi$ is an element of the stabiliser of $\cQ$ and 
$$ 
\pi(x)  = \frac{ax+b}{cx+d} = \frac{(ax+b)(cx+d)}{(cx+d)^2}, 
$$
the image of $f$ satisfies $Im(f) \subset \cQ\cup\{0\}$. 

If $c=0$, then $a\neq 0$ and $d\neq 0$.  Without loss of generality, we let $d=1$ and $f(x)=ax^2+b$. Assume that $b\neq 0$, then $f(x)$ is not a square of a polynomial as $q$ is odd. From Theorem \ref{thmWeilBound} we have 
$$\left\vert \sum_{x\in\gf(q)}\eta(f(x)) \right\vert \le \sqrt{q}, $$
where $\eta$ is the quadratic character on $\gf(q)$. However, since $\pi(x) = ax+b$ is an element of the stabiliser of $\cQ$, we see that $f(x) = ax^2+b\in \cQ$ for $x \neq 0$. This gives
\begin{equation}\label{eq3}
 \sqrt{q}\ge \left\vert \sum_{x\in\gf(q)}\eta(f(x)) \right\vert \ge \left\vert \sum_{x\in\gf(q),x\neq  0}\eta(f(x)) \right\vert-\vert \eta(b)| \ge q-2,
\end{equation}
which is a contradiction as $q>9$.

Thus in this case we must have $b = 0$, which leads to $a\in\cQ$, as $\pi(x) = ax $ fixes $\cQ$.

If $c\neq 0$, we further discuss the value of $a$. If $a = 0$, by similar arguments as above, we can let $b=1$ and have the conclusions of $d = 0$ and $c\in \cQ$, i.e., $\pi(x) = b/x $ with $b\in\cQ$. Next we assume $a\neq 0$. Let $b' = b/a $ and  $ d' = d/c$, we have 
$$f(x) = ac(x^2+b')(x^2+d').$$
Since $ad \neq bc$ we have $x^2+b'\neq x^2+d'$. Thus $f(x)$ is not a square of a polynomial since $q$ is odd. Then from Theorem \ref{thmWeilBound} we have 
$$\left\vert \sum_{x\in\gf(q)}\eta(f(x)) \right\vert \le 3 \sqrt{q}, $$
where $\eta$ is the quadratic character on $\gf(q)$. Similar with \eqref{eq3} we have 
$$ 3\sqrt{q}\ge \left\vert \sum_{x\in\gf(q)}\eta(f(x)) \right\vert \ge \left\vert \sum_{x\in\gf(q),x\neq  0}\eta(f(x)) \right\vert  - |\eta(bd)|\ge q-2,$$
which is a contradiction when $q\ge 13$. For $q = 11$, we verify that the stabiliser of $\cQ$ is indeed the one given in the theorem. Summarizing the results above yields the desired conclusions on the stabiliser of $\cQ$. 

2) The conclusion on the stabiliser of $V_0$ is proved in Theorem A of \cite{Iwasaki}. Then the desired 
conclusion on the stabiliser of $U_0$ follows from that of $V_0$ and Lemma \ref{lemmaStabTrans}.

\end{proof}

We make the following remarks: 
\begin{itemize}
\item The stabilisers of $V_0$, $U_0$ and $V_i \triangle V_j$ in $\PSL_2(q)$ for $q \equiv 3 \pmod{4}$  are given in \cite{Iwasaki2}, where $V_i := V_0+i$. Notwithstanding that our base block $B_\tau$ equals $V_0\triangle V_\tau$, we still need to consider its stabiliser in $\GA_1(q)$, which is contained in $\PGL_2(q)$ but not in $\PSL_2(q)$, and for $q\equiv 1 \pmod{4}$ either.

\item The stabilisers of $U\cup\{\infty\}$ in $\PSL_2(q)$ and $\PGL_2(q)$ are given in \cite{Iwasaki}, where $U$ is a subgroup of $(\gf(q)^*, \times)$. For the stabiliser of $U$ in these two groups, the circumstance becomes a little complicated. An existing method for solving this problem is to consider the cardinalities of the orbits of subgroups in $\PSL_2(q)$, as considered in \cite{Cameron} and \cite{Liuetal}. 

\item Here we develop another method from the perspective of character sums, which is more concise. Note that both methods may not work when the size $f$ of the subgroup $U$ is small. 

\item For $q=9$, the stabiliser of $\cQ$ is equivalent to $S_4$ in $\PSL_2(q)$.
\end{itemize}

Below we introduce parameters of the $2$-design derived from the action of $\GA_1(q)$ on the base block $B_\tau = \cQ\triangle (\cQ+\tau)$. 

\begin{theorem}\label{thm-mainr3} 
Let $q>9$ be a power of an odd prime, and let $\cQ$ denote the set of all nonzero squares in $\gf(q)$. 
Define $\cP=\gf(q)$ and 
\begin{eqnarray*}
B_\tau= \cQ\triangle(\cQ+\tau),  
\end{eqnarray*} 
where $\cQ+\tau =\{x+\tau: x \in \cQ\}$ and $\tau \in \gf(q)^*$. 
Then $(\cP, \GA_1(q)(B_\tau))$ is a $2$-$(q, k, \lambda)$ design, where 
\begin{eqnarray*}
k=\left\{ 
\begin{array}{ll}
\frac{q+1}{2} & \mbox{ if } q \equiv 3 \pmod{4}, \\
\frac{q+3}{2} & \mbox{ if } q \equiv 1 \pmod{4} \mbox{ and } \tau \in \cQ, \\
\frac{q-1}{2} & \mbox{ if } q \equiv 1 \pmod{4} \mbox{ and } \tau \in \cN  \\
\end{array} 
\right. 
\end{eqnarray*} 
and 
\begin{eqnarray*}
\lambda=\left\{ 
\begin{array}{ll}
k(k-1) & \mbox{ if } q \equiv 3 \pmod{4}, \\
\frac{k(k-1)}{2} & \mbox{ if } q \equiv 1 \pmod{4}. 
\end{array} 
\right. 
\end{eqnarray*} 

\end{theorem} 

\begin{proof}
The block size of $B_\tau$ follows directly from Lemma \ref{lemmaCycNum2}. Next we determine the stabiliser of $B_\tau$ in $\GA_1(q)$.

It is easy to check that 
\begin{align*}
\{0,\tau\}\cap B_\tau= \begin{cases}
\{\tau\} &\mbox{ if }~\tau\in\cQ,~q\equiv 3\pmod{4}, \\
\{0\} &\mbox{ if }~\tau\in\cN,~q\equiv 3\pmod{4},\\
\{0,\tau\} &\mbox{ if }~\tau\in\cQ,~q\equiv 1\pmod{4}, \\
\emptyset &\mbox{ if }~\tau\in\cN,~q\equiv 1\pmod{4}.\\
\end{cases}
\end{align*}

Let $\pi(x) =\displaystyle \frac{\tau-x}{x}\in\PGL_2(q)$. Then from the conclusion of $\{0,\tau\}\cap B_\tau$ and the cardinality of $B_\tau$, we have 
\begin{align*}
 B^\pi_\tau= \begin{cases}
\cQ\cup\{0\} &\mbox{ if }~\tau\in\cQ,~q\equiv 3\pmod{4}, \\
\cQ\cup\{\infty\} &\mbox{ if }~\tau\in\cN,~q\equiv 3\pmod{4}, \\
\cQ\cup\{0,\infty\} &\mbox{ if }~\tau\in\cQ,~q\equiv 1\pmod{4}, \\
\cQ &\mbox{ if }~\tau\in\cN,~q\equiv 1\pmod{4}.\\
\end{cases}
\end{align*}

Let $\widetilde{G} = \PGL_2(q)$, $A:=\{\pi(x) = ax, ~a\in\cQ\}$ and $B:=\{\pi(x) = b/x,~b\in\cQ\}$. From Theorem \ref{thmStabQ} and Lemma \ref{lemmaStabTrans} we have 

\begin{align*}
 \widetilde{G}_{B_\tau} = \pi^{-1}  \widetilde{G}_{B^\pi_\tau} \pi= \begin{cases}
A^\pi &\mbox{ if }~\tau\in\cQ,~q\equiv 3\pmod{4}, \\
(A\cup B)^\pi &\mbox{ if }~\tau\in\cQ,~q\equiv 1\pmod{4},
\end{cases}
\end{align*}
where $A^\pi = \pi^{-1}A\pi$.

Let $G = \GA_1(q)$. Since $\GA_1(q)\subset \PGL_2(q)$, we have $G_{B_\tau} = G\cap \widetilde{G}_{B_\tau}$, witch leads to 
\begin{align*}
 {G}_{B_\tau} =   \begin{cases}
\{\mathbf{1}(x) = x\} &\mbox{ if }~~q\equiv 3\pmod{4}, \\
\{\mathbf{1}(x) = x,\,\pi(x) = \tau-x\} &\mbox{ if }~ ~q\equiv 1\pmod{4}.
\end{cases}
\end{align*}
The desired conclusions on $\lambda$ then follow from Corollary \ref{cor-ourmaincorr}. 

\end{proof}

\section{The third family of $2$-designs} 

In this section, we consider several constructions of $2$-designs by the action of $\GA_1(q)$ for even $q$. First we consider the 
$2$-design  obtained from the action of $\GA_1$ on the base block $B:=\{u\in\gf(q),~\Tr(u^3)=1\}$ whose 
cardinality is given in the following lemma. 

\begin{lemma}\label{lemmaCardTru3}
Let $q=2^m$, where $m \geq 4$ and $m$ is even. Define 
$$
B=\{u \in \gf(q): \tr(u^3)=1 \}. 
$$ 
We have $|B| = 2^{m-1}+(-2)^{m/2}$.
\end{lemma}
\begin{proof}
Define $e(x) = (-1)^{\tr(x)}$ for $x\in\gf(q)$. It was proved in \cite{DG70} and \cite{Carlitz} that 
$$\sum_{x\in \gf(q)}e(x^3) = (-2)^{m/2+1}.$$ 
Consequently, 
\begin{align*}
|B| &= |\{u \in \gf(q): \tr(u^3)=1 \}| \\
&=\frac{1}{2}\sum_{x\in\gf(q)}\left(1-e(x^3)\right)\\
&=\frac{1}{2}(q-(-2)^{m/2+1})\\
&=2^{m-1}+(-2)^{m/2}.
\end{align*}
\end{proof}

Next we determine the stabiliser of $B$ in $\GA_1(q)$.
To this end, we need the following lemma.

\begin{lemma}\label{lemmaTrBspan}
Let $q=2^m$, where $m \geq 3$. Define  
$$
B=\{u \in \gf(q): \tr(u^3)=1 \}. 
$$ 
The linear expansion of $B$ over $\gf(2)$ is the whole space $\gf(q)$.

\end{lemma}

\begin{proof}
Let $v$ be an element in $\gf(q)\setminus B$, i.e., $\Tr(v^3)=0$. We need to show that $v$ is the sum of some 
elements of $B$. If $m$ is odd, we have $\Tr(1) = 1$ and $\Tr((v - 1)^3) = 1$, which shows that $v = (v-1)+1$ 
is the sum of the two elements of $B$. 

Next we consider the case that $m$ is even. 
Let $u$ be an element in $B$, then we have 
\begin{align*}
\Tr\left((v+u)^3\right) =& \Tr(v^3) + \Tr(v^2u+u^2v) + \Tr(u^3)=\Tr(v^2(u+u^4)) + 1.
\end{align*}
Note that $m$ is even. Let $\omega$ be a $3rd$ root of unit in $\gf(q)$. Then we have $\omega u,\omega^2u \in B$ and 
\begin{align*}
\Tr\left((v+\omega u)^3\right)&=\Tr(v^2(\omega u+\omega^4u^4)) + 1
=\Tr\left(\omega v^2( u+u^4)\right) + 1, \\
\Tr\left((v+\omega^2 u)^3\right)&=\Tr(v^2(\omega^2 u+\omega^8u^4)) + 1
=\Tr\left(\omega^2 v^2( u+u^4)\right) + 1. 
\end{align*}
Summing up the three equations above, we get 
$$\Tr\left((v+u)^3\right)+\Tr\left((v+\omega u)^3\right)+\Tr\left((v+\omega^2 u)^3\right) =1,$$
which means that $v+\omega^ku\in B$ for some $k\in \{0,1,2\}$. Thus,  $v=(v+\omega^ku)+\omega^ku$. 
This shows that $v$ is the sum of the two elements $v+\omega^ku$ and $\omega^ku$ in $B$.  
The proof is then completed. 
\end{proof}

The following result will be useful in determining the stabiliser of $B$.

\begin{corollary}\label{corU3hasTr01}
Let $q=2^m$, where $m \geq 3$, and let  
$$
B=\{u \in \gf(q): \tr(u^3)=1 \}. 
$$ 
Then for any $t\in\gf(q)^*$, there exists $x_1\in B$ such that $\Tr(tx_1)=1$.
\end{corollary} 

\begin{proof}
Suppose on the contrary that for some $t\in\gf(q)^*$, there is no such $x_1\in B$ such that $\Tr(tx_1)=1$. Then we have $tB\subset T_0:=\{u\in\gf(q),~\Tr(u) = 0\}$. Let $L_B$ be the linear subspace spanned by the elements of $B$ when $\gf(q)$ is viewed as a vector space over $\gf(2)$. Since $T_0$ is a linear subspace, we must also have $tL_B\subset T_0$. According to 
Lemma \ref{lemmaTrBspan}, we have $L_B = \gf(q)$, which contradicts $tL_B\subset T_0$ since $t\neq 0$. Thus we have proved the desired conclusion.
\end{proof}

The stabiliser of $B$ in $\GA_1(q)$ is depicted in the following theorem.

\begin{theorem}\label{thmStabTru3}
Let $q=2^m$, where $m \geq 4$ and $m$ is even. Define 
$$
B=\{u \in \gf(q): \tr(u^3)=1 \}. 
$$ 
The stabiliser of $B$ in $G=\GA_1(q)$ is given by
$$G_B = \{\pi(x) = \omega^kx+\delta\omega^j:~k,\, j\in\{ 0,1,2\},~\delta \in\{0,1\}\},$$
where $\omega$ is the 3rd root of unit in $\gf(q)$, i.e., $\omega^3=1$.
\end{theorem}

\begin{proof}
We prove the conclusion in two steps. First, we show that the elements in $G_B$ given above fix $B$.
Let $u\in B$ and $\pi(x) =  \omega^kx+\epsilon\omega^j$ be an element in $G_B$. We need to show that $\pi(u) \in B$, or equivalently, $\Tr(\pi(u)^3) = 1$.

If $\delta = 0$, we have $\pi(u) = \omega^ku$ and $\Tr(\omega^{3k}u^3) = \Tr(u^3) = 1$. If $\delta = 1$, we have 
\begin{align*}
\Tr(\pi(u)^3)& =\Tr\left((\omega^ku+\omega^j)^3\right)\\
&=\Tr(\omega^{3k}u^3+\omega^{2k+j}u^2+\omega^{k+2j}u+\omega^{3j})\\
&=\Tr(u^3) + \Tr(\omega^{2k+j}u^2)+\Tr(\omega^{k+2j}u)+\Tr(1)\\
&=1+\Tr(\omega^{2k+j}u^2+\omega^{2k+4j}u^2)+m\\
&=1.
\end{align*}
Thus $G_B$ fixes $B$. 

Next, we prove that $G_B$ is indeed the whole stabiliser group of $B$ in $\GA_1(q)$. Suppose there exists another element $\pi(x) = ax+b$ in $\GA_1(q)$ that fixes $B$. Then for any $u\in B$ we have 
\begin{align}
\Tr(\pi(u)^3)& =\Tr\left((au+b)^3\right)\nonumber\\
&=\Tr(a^{3}u^3+a^2bu^2+ab^2u+b^3)\nonumber\\
&=\Tr(a^3u^3) + \Tr\left(a^2(b+b^4)u^2\right)+\Tr(b^3) =1.\label{eq4}
\end{align}
Since we also have $\omega u\in B$ and $\omega^2 u\in B$ , similarly we have 
\begin{equation}\label{eq5}
\Tr(a^3u^3) + \Tr\left(a^2(b+b^4)\omega^2u^2\right)+\Tr(b^3) =1
\end{equation}
and 
\begin{equation}\label{eq6}
\Tr(a^3u^3) + \Tr\left(a^2(b+b^4)\omega u^2\right)+\Tr(b^3) =1.
\end{equation}
Taking the difference of Equations \eqref{eq5} and \eqref{eq6}, we reach at 
$$\Tr\left(a^2(b+b^4)u^2\right) = 0,~\forall u\in B. $$
Notice that the Frobenius automorphism $F(u) = u^2$ also fixes $B$. The equation above can be written as 
$$\Tr\left(a^2(b+b^4)u\right) = 0 ,~\forall u\in B. $$
To avoid a contradiction with Corollary \ref{corU3hasTr01},  we must have 
$$a^2(b+b^4) = 0.$$
Since $a\neq 0$, we see that $b = 0 $ or $b = \omega ^ k$ for some $k\in \{0,1,2\}$. This gives us $\Tr(b^3) = 0$ and \eqref{eq4} becomes 
$$\Tr(a^3u^3) = 1 ,~\forall u\in B.$$
This means that the permutation $\pi'(x) = ax$ fixes $B$. Assume the multiplicative order of $a$ in $\gf(p)$ is $f:=\ord(a)$, then the orbits of $\pi'$ acting on $\gf(q)$ is composed of $\{0\}$ and $(2^m-1)/f$ orbits of size $f$, which are the cosets $\{\C^{(2^m-1)/f}_k,~k =  0,1,2,..,f-1\}$. It is well known that if $\pi'$ fixed $B$, then $B$ must be composed of orbits of $\pi'$. Since $0\notin B$, $f$ must divide $|B|=2^{m-1}+(-2)^{m/2}$. Combining this with the fact that $f|(2^m-1)$, we have 
$$f\vert\text{gcd}(2^{m-1}+(-2)^{m/2},2^m-1) = 3.$$
This means that $a=\omega^k$ for some $k = \{0,1,2\}$, which completes our proof.
\end{proof}

As a direct corollary of Lemma \ref{lemmaCardTru3} and Theorem \ref{thmStabTru3}, the parameters of  the design from $B$ is given as follows.

\begin{theorem}\label{thm-mainr4} 
Let $q=2^m$, where $m \geq 4$ and $m$ is even. Define $\cP=\gf(q)$ and $\cB = \GA_1(q)(B)$, where
$$
B=\{u \in \gf(q): \tr(u^3)=1 \}. 
$$
Then $(\cP, \cB)$ is a $2$-$(2^m,\, k,\, k(k-1)/12)$ design, where $k=2^{m-1}+(-2)^{m/2}$.  
\end{theorem} 

Next we determine the parameters of the $2$-design of the action of $\GA_1(q)$ on the base block 
$B_2 =\{u \in \gf(q): \tr(u^3-u)=1\}$. The cardinality of $B_2$ is given in the following lemma. 

\begin{lemma}\label{lemmaCardTru3u}
Let $q=2^m$, where $m \geq 3$. Define  
$$
B=\{u \in \gf(q): \tr(u^3-u)=1 \}. 
$$ 
We have 
\begin{align*}
|B_2| = \begin{cases}
2^{m-1}-\left(\frac{2}{m}\right)\cdot 2^{(m-1)/2} &\mbox{ if }~m \equiv 1\pmod{2}, \\
2^{m-1}+2^{m/2+1} &\mbox{ if }~m\equiv 0\pmod{4}, \\
2^{m-1} &\mbox{ if }~m\equiv 2\pmod{4},  
\end{cases}
\end{align*} 
where $\left(\frac{\cdot}{\cdot}\right)$ is the Jacobi symbol. 
\end{lemma}

\begin{proof}
Define $e(x) = (-1)^{\tr(x)}$ for $x \in \gf(q)$.
The cardinality of $B_2$ is given by  
\begin{align}\label{eqn-june171}
|B_2| = |\{u \in \gf(q): \tr(u^3-u)=1 \}| 
=\frac{1}{2}\sum_{x\in\gf(q)}\left(1-e(x^3-x)\right)
\end{align}
By Theorems 1 and 2 in \cite{Carlitz},   
\begin{align}\label{eqn-june172}
\sum_{x\in \gf(q)}e(x^3-x) =
\begin{cases}
\left(\frac{2}{m}\right)\cdot 2^{(m+1)/2} ~&\mbox{ if } m\equiv 1\pmod{2}, \\
(-2)^{m/2+1} ~&\mbox{ if } m\equiv 0\pmod{4}, \\
0 ~&\mbox{ if } m\equiv 2\pmod{4}.
\end{cases}
\end{align}
Combining (\ref{eqn-june171}) and (\ref{eqn-june172}) yields the desired results.
\end{proof}

To determine the stabiliser of $B_2$ in $\GA_1(q)$, we need the following lemma.
\begin{lemma}\label{lemmaEqtionsNbound}
Let $q=2^m$, where $m \geq 3$. Denote the number of solutions to the following equations 
\begin{align*}
\begin{cases}
\Tr(x^3-x) =u\\
\Tr(ax) = v
\end{cases}
\end{align*}
in $\gf(q)$ by $N(u,v)$, where $a\in\gf(q)^*$ and $u,v\in\gf(2)$. We have $N(u,v)\le 2^{m-2}+2^{m/2-1}$.
\end{lemma}
\begin{proof}
Let $f(x) = \Tr(x^3)$. It is well known that the quadratic form $f(x)$ has rank $m-1$ when $m$ is odd 
and rank $m-2$ when $m$ is even. The desired conclusion then follows from Propositions 3.1, 3.2, 
and 3.3 in \cite{Klapper}. 
\end{proof}

We now describe the stabiliser of $B_2$ in $\GA_1(q)$ with the following theorem.

\begin{theorem}\label{thmStabTru3u}
Let $q=2^m$, where $m \geq 3$ and 
$$
B_2=\{u \in \gf(q): \tr(u^3-u)=1 \}. 
$$ 
Then the stabiliser of $B_2$ in $G = \GA_1(q)$ is  
\begin{align*}
G_{B_2} = \begin{cases}
\{ \mathbf{1}(x) = x,~\pi_0(x) = x+1 \} &\mbox{ if } ~m \not \equiv 0 \pmod{4}, \\
\left\{ \pi(x) = \omega^kx+\alpha_j\omega^{-k}, \ k\in\{1,2\}, \ j\in\{1,2,3,4\}\right\}\\
\quad\cup\left\lbrace\mathbf{1}(x) = x ,~\pi(x) = x+\omega^k,~k\in\{0,1,2\}\right\rbrace &\mbox{ if } ~ m\equiv 0 \pmod{4},
\end{cases}
\end{align*}
where $\omega$ is a 3$rd$ root of unit in $\gf(q)$ and $\left\lbrace\alpha_j,\, j \in \{1,2,3,4\}\right\rbrace$ are the roots of $x^4+x+1=0$ in $\gf(q)$.
\end{theorem}

\begin{proof}
First, we conclude that $\pi_0(x) = x+1$ fixes $B_2$ as  
$$\Tr((u+1)^3+(u+1)) = \Tr(u^3+u^2+u+u) = \Tr(u^3+u).$$
Next we assume that $\pi(x) = ax+b$ fixes $B_2$. Then $\pi':=\pi\circ\pi_0(x) = \pi(x+1)$ also fixes $B_2$. This gives us 
\begin{align}
\Tr(\pi(u)^3 +\pi(u)) &= \Tr\left((au+b)^3+(au+b)\right)\nonumber\\
&= \Tr(a^3u^3+a^2bu^2+ab^2u+b^3+au+b)\nonumber\\
&=\Tr(a^3u^3+a^2(1+b+b^4)u^2+b^3+b)=1\label{eq7}
\end{align}
and 
\begin{align}
& \Tr(\pi'(u)^3 +\pi'(u)) -\Tr(\pi(u)^3 +\pi(u)) \nonumber \\ 
&= \Tr\left((au+a+b)^3+(au+a+b)\right)-\Tr\left((au+b)^3+(au+b)\right)\nonumber\\
&= \Tr\left(a(au+b)^2+a^2(au+b)+a^3+a\right)\nonumber\\
&=\Tr\left((au+b)(a^{1/2}+a^2)+a^3+a\right)=0\label{eq8}
\end{align}
for any $u\in B_2$. Since $\pi (x) = ax +b $ is a permutation of $B_2$, Equation \eqref{eq8} can be written as 
\begin{equation}\label{eq9}
\Tr\left((a^{1/2}+a^2)u+a^3+a\right)=0,~\forall u \in B_2.
\end{equation}

Let $N_0$ be the number of solutions to the following equations
\begin{align*}
\begin{cases}
\Tr\left((a^{1/2}+a^2)x+a^3+a\right)=0\\
\Tr(x^3-x) = 1.
\end{cases}
\end{align*}
Equation \eqref{eq9} gives that $N_0 = |B_2|$, while Lemma \ref{lemmaEqtionsNbound} says that $N_0\le 2^{m-2} + 2^{m/2-1}$ when $a^{1/2}+a^2 \neq 0$. Since $|B_2|>2^{m-2} + 2^{m/2-1}$ by Lemma \ref{lemmaCardTru3u}, to avoid a contradiction we must have  $a^{1/2}+a^2 = 0$. 

When $m$ is odd, $a^{1/2}+a^2 = 0$ is equivalent to $a=1$ as $a\neq 0$. Then \eqref{eq7} becomes
\begin{equation}\label{eq10}
\Tr\left(u^3+u + (b^{1/2}+b^2)u + b^3+b\right) =1+ \Tr\left((b^{1/2}+b^2)u + b^3+b\right) = 1
\end{equation}
for any $u\in B_2$. By the same arguments as above, we will have $b=0$ or $1$, which leads to the desired conclusion for odd $m$.

When $m$ is even, $a^{1/2}+a^2 = 0$ is equivalent to $a=\omega^k$ since $a\neq 0$. Then \eqref{eq9} becomes $\Tr(a^3+a) = \Tr(\omega^k) = 0$. Thus when $m \equiv 2 \pmod{4}$, we must have $k=0$ and $a=1$. Then again we see that \eqref{eq7} becomes \eqref{eq10} and same arguments lead to $b = 0$ or $1$, which is our desired conclusion. When $m\equiv 0\pmod{4}$, Equation \eqref{eq7} becomes
$$\Tr\left(\omega^{2k}(\omega^{-2k}+1+b+b^4)u^2+b^3+b\right) = 0$$
for any $u\in B_2$. With the same arguments as we gave for $a$, we must have $\omega^{-2k}+1+b+b^4=0$, which leads to 
\begin{align*}
b = \begin{cases} 
\omega^{-k}\alpha_j &\mbox{ if } ~k = 1,2, \\
\omega^{l}\delta  &\mbox{ if } ~k=0,
\end{cases}
\end{align*}
where $\delta \in\{ 0,1\}$,  $j\in\{0,1,2,3\}$ and $l\in\{0,1,2\}$. Note that $\alpha_j\in\gf(q)$ since $4|m$, and for these values of $b$, it is straightforward to verify that $\Tr(b^3+b)=0$ and \eqref{eq7} is satisfied, which means $\pi(x) = ax+b$ fixes $B_2$. Hereby we complete our proof.
\end{proof}

As a direct corollary of Lemma \ref{lemmaCardTru3u} and Theorem \ref{thmStabTru3u}, we have the following conclusion.

\begin{theorem}\label{thm-mainr5} 
Let $q=2^m$, where $m \geq 3$. Define $\cP=\gf(q)$ and 
$$
B_2=\{u \in \gf(q): \tr(u^3-u)=1 \}. 
$$ 
Then $(\cP, \cB_2)$ is a $2$-$(2^m,\, k,\, \lambda)$ design, where 
\begin{align*}
k = \begin{cases}
2^{m-1}-\left(\frac{2}{m}\right)\cdot 2^{(m-1)/2} &\mbox{ if } ~m \equiv 1\pmod{2}, \\
2^{m-1}+2^{m/2+1} &\mbox{ if } ~m\equiv 0\pmod{4}, \\
2^{m-1} &\mbox{ if } ~m\equiv 2\pmod{4},  
\end{cases}
\end{align*}
and 
\begin{align*}
\lambda = \begin{cases}
k(k-1)/2 &\mbox{ if } ~m \not\equiv 0\pmod{4}, \\
k(k-1)/12 &\mbox{ if } ~m\equiv 0\pmod{4},
\end{cases}
\end{align*}
where $\left(\frac{\cdot}{\cdot}\right)$ is the Jacobi symbol.  
\end{theorem}

It would be interesting to settle the following conjecture, which is confirmed by Magma for 
$m \in \{3, 5, 6, 7\}$. 

\begin{conj}\label{conj-june251} 
Let $m \geq 3$ such that $m \not\equiv 0 \pmod{4}$. Then the pair $(\cP, \cB_2)$ is a $3$-design. 
\end{conj}

\section{The fourth family of $2$-designs} 

In this section we consider the $2$-design derived from the action of $\GA_1(q)$ on the base block $B_j:=\{u\in\gf(q) : \Tr(u) = j\}$ for $j\in\gf(p)$, where $q=p^m$. The cardinality of $B$ is clearly $p^{m-1}$. We need to 
determine the stabiliser of $B_j$ in the group $\GA_1(q)$.

\begin{theorem}\label{thmStabTru}
Let $q=p^m$ be a prime power and
$$
B_j=\{u \in \gf(q): \tr(u)=j \}
$$ 
for $j\in\gf(p)$. The stabiliser of $B_j$ in $G= \GA_1(q)$ is 
$$G_{B_j} = \{\pi(x) = ax+b : ~a\in\gf(p)^*, \ b\in B_{j-ja}\}$$ 
for $j\in\gf(p)$, and $|G_B| = (p-1)p^{m-1}$.
\end{theorem}

\begin{proof}
Let $u_1\in\gf(q)$ with $\Tr(u_1)=1$. Then we have 
$$B_j = B_0 + ju_1, \forall j \in \gf(p).$$
Suppose $\pi(x) = ax+b\in\GA_1(q)$ fixes $B_j$, which is equivalent to  
$$aB_j +b = aB_0 + aju_1 +b = B_0+ju_1 = B_j.$$ 
Consequently, 
\begin{equation}\label{eq12}
aB_0 = B_0 - (a-1)ju_1-b.
\end{equation}

Since $aB_0$ is a linear subspace, we see that $B_0 - (a-1)ju_1-b$ is also a linear subspace of $\gf(q)$, which is the case if and only if $-(a-1)ju_1-b\in B_0$, or equivalently, $\Tr\left((a-1)ju_1+b\right) = 0$. Then 
\eqref{eq12} becomes  $aB_0 = B_0$, which means $\pi'(x) = ax $ fixes $B_0$. 

Let $\mathcal{F} := \{f(a), ~f\in\gf(p)[x]\}\subset \gf(q)$ be the minimal finite field that contains $a$. Since $aB_0 = B_0$, we have $\mathcal{F}u_0\subset B_0$ for any $u_0\in B_0$. If $\mathcal{F}\neq \gf(p)$, then there must exist $0\neq u_0\in\mathcal{F}$ such that $\Tr(u_0) = 0$, i.e., $u_0\in B_0$. This leads to $\mathcal{F}u_0 = \mathcal{F} \subset B_0$, which means every element in $\mathcal{F}$ has trace 0. This contradicts to our assumption of $\mathcal{F}\neq \gf(p)$. Thus we must have $a\in \mathcal{F} =\gf(p)$. Then from $\Tr\left((a-1)ju_1+b\right) = 0$, we have $\Tr(b) = j-ja$, where $a\in\gf(p)^*$.
\end{proof}

Combining Theorem \ref{thmStabTru} and Corollary \ref{cor-ourmaincorr}, we immediately have the following conclusion.

\begin{theorem}\label{thm-2geometryd}
Let $q=p^m$, where $m \geq 3$ and $p$ is a prime. Define $\cP=\gf(q)$ and $\cB = \GA_1(q)(B)$, where
$$
B=\{u \in \gf(q): \tr(u)=j\}  
$$ 
with $j\in\gf(p)$.
Then $(\cP, \cB)$ is a $2$-$(p^m,\, p^{m-1},\, (p^{m-1}-1)/(p-1))$ design. 
\end{theorem} 

The $2$-designs documented in Theorem \ref{thm-2geometryd} are in fact the $2$-design formed by all the 
$(m-1)$-flats in the affine geometry $\AG(m,p)$. Our objective here is to show that the geometric $2$-design formed by the 
$(m-1)$-flats in $\AG(m,p)$ can be obtained by the action of $\GA_1(q)$, and has a simpler expression given 
in Theorem \ref{thm-2geometryd}. We remark that the pair $(\cP, \cB)$ is a $3$-design when $p=2$ and $m \geq 3$. 

\section{The classical linear codes of the $2$-designs of this paper}\label{sec-codesdesigns} 

A $[v, \kappa, d]$ code $\C$ over $\gf(p)$ is a linear subspace of $\gf(p)^v$ with dimension $\kappa$ 
and minimum Hamming distance $d$. Let $A_i:=A_i(\C)$, which denotes the
number of codewords with Hamming weight $i$ in $\C$, where $0 \leq i \leq v$. The sequence 
$(A_0, A_1, \cdots, A_{v})$ is
called the \textit{weight distribution} of $\C$, and $\sum_{i=0}^v A_iz^i$ is referred to as
the \textit{weight enumerator} of $\C$.    

Let $\bD = (\cP, \cB)$ be a $t$-$(v, k, \lambda)$ design with $b \ge 1$ blocks. 
The points of $\cP$ are usually indexed with $p_1,p_2,\cdots,p_v$, and the blocks of $\cB$ 
are normally denoted by $B_1, B_2, \cdots, B_b$. The {\em incidence matrix\index{incidence 
matrix}} $M_\bD=(m_{ij})$ of $\bD$ is a $b \times v$ matrix where $m_{ij}=1$ if  $p_j$ is on $B_i$ 
and $m_{ij}=0$ otherwise. The binary matrix $\cB$ is viewed as a matrix over $\gf(p)$ for any 
prime $p$, and its row vectors span a linear code of length $v$ over $\gf(p)$, which is 
denoted by $\C_p(\bD)$ and called the \emph{classical code} of $\bD$ over $\gf(p)$ \cite{AK92,Tonc93,Tonchev,Tonchevhb}. 

We do not plan to study the classical codes of the designs documented in this paper. Our objectives 
of this section are the following: 
\begin{enumerate}
\item To demonstrate that the codes of some of the $2$-designs of this paper are optimal or 
      have best parameters known. 
\item To justify that the construction and study of $2$-designs could be very interesting from 
      a coding theoretic point of view, though it is more interesting to construct and study 
      $t$-designs for larger $t$ in combinatorics. 
\item To propose a few open problems regarding some of the $2$-designs documented in this paper.             
\end{enumerate}

To achieve the objectives above, we consider only the designs in Theorems \ref{thm-mainr4} and \ref{thm-mainr5}. 
For odd $m$, we have the following conjecture. 

\begin{conj}\label{conj-june252}  
Let $m \geq 3$ be odd. Let $\bD=(\cP, \cB_2)$ be the $2$-design in Theorem \ref{thm-mainr5}. 
Let $\C_2(\bD)$ be the binary code of the design $\bD$. Then $\C_2(\bD)$ has parameters 
$[2^m,\, 2m+1,\, 2^{m-1}-2^{(m-1)/2}]$ and weight enumerator 
$$ 
1+uz^{2^{m-1}-2^{(m-1)/2}}+vz^{2^{m-1}}+uz^{2^{m-1}+2^{(m-1)/2}}+z^{2^m}, 
$$ 
where 
$$ 
u=2^{2m-1}-2^{m-1} \mbox{ and } v = 2^{2m}+2^m-2. 
$$ 
In addition, the dual code $\C_2(\bD)^\perp$ has parameters $[2^m, 2^m-1-2m, 6]$.  
\end{conj} 

Conjecture \ref{conj-june252} is confirmed by Magma for $m \in \{3, 5, 7\}$. In all these three cases, 
the linear code $\C_2(\bD)$ is optimal\footnote{For the meaning of optimality and the justification of 
optimality of the codes, see http://www.codetables.de}. If Conjecture \ref{conj-june252} is true, then $\C_2(\bD)$ 
holds three $3$-designs, and $\C_2(\bD)^\perp$ holds exponentially many $3$-designs (see \cite{DingLi171} 
for detail).

\begin{table}[ht]
\begin{center} 
\caption{Conjectured weight distribution of $\C_2(\bD)$ for even $m \geq 4$.}\label{tab-KasamievenmExtDual}  
{
\begin{tabular}{ll}
\hline
Weight $w$    & No. of codewords $\overline{A}^\perp_w$  \\ \hline
$0$                                                        & $1$ \\
$2^{m-1}-2^{m/2}$           &  $(2^m-1)2^{m-2}/3$ \\
$2^{m-1}-2^{(m-2)/2}$           &  $(2^m-1)2^{m+1}/3$ \\
$2^{m-1}$           &              $(2^m-1)(2^{m-1}+2)$         \\
$2^{m-1}+2^{(m-2)/2}$           &  $(2^m-1)2^{m+1}/3$\\
$2^{m-1}+2^{m/2}$           & $(2^m-1)2^{m-2}/3$ \\ 
$2^m$                       & $1$ \\ \hline
\end{tabular}
}
\end{center}
\end{table}

\begin{conj}\label{conj-june253}  
Let $m \geq 4$ be even. Let $\bD=(\cP, \cB_2)$ be the $2$-design in Theorem \ref{thm-mainr5} 
or let $\bD=(\cP, \cB)$ be the $2$-design in Theorem \ref{thm-mainr4}. 
Let $\C_2(\bD)$ be the binary code of the design $\bD$. Then $\C_2(\bD)$ has parameters 
$[2^m,\, 2m+1,\, 2^{m-1}-2^{m/2}]$ and the weight distribution in Table \ref{tab-KasamievenmExtDual}. In addition, the dual code $\C_2(\bD)^\perp$ has parameters $[2^m,\, 2^m-1-2m,\, 6]$.  
\end{conj} 

Conjecture \ref{conj-june253} is confirmed by Magma for $m \in \{4, 6, 8\}$. When $m=4$, the code $\C_2(\bD)$ 
is optimal. When $m=6$ and $m=8$, the code $\C_2(\bD)$ has the best parameters known (see http://www.codetables.de). If Conjecture \ref{conj-june253} is true, then $\C_2(\bD)$ 
holds four $2$-designs, and $\C_2(\bD)^\perp$ holds exponentially many $2$-designs for even $m$ (see \cite{DingZhou171} for detail). 

We remark that the linear code $\C_p(\bD)$ of the design $\bD$ in Theorem \ref{thm-2geometryd} should have 
parameters 
$$\left[p^s, \  \binom{p+s-1}{s},\  p^{s-1}\right].$$ 
A proof of this result may be found in \cite{AK92}. 

The examples of codes above demonstrate that it is worthwhile to construct and study $2$-designs, as $2$-designs 
may yield optimal linear codes. Hence, $t$-designs with small $t$ are also interesting, 
not to mention their applications in other areas of mathematics and engineering.

\section{Summary and concluding remarks} 

The contributions of this paper are the construction of the five infinite families of 
$2$-designs and the determination of their parameters, which are documented in 
Theorems \ref{thm-mainr1},  
\ref{thm-mainr2}, 
\ref{thm-mainr3},   
\ref{thm-mainr4}, and  
\ref{thm-mainr5}. 
Another contribution of this paper is a different representation of the $2$-design formed 
by the $(m-1)$-flats in $\AG(m, q)$, which was documented in Theorem \ref{thm-2geometryd}.   

Though the construction of $t$-designs with group action is a standard approach, selecting a proper 
point set $\cP$, a suitable permutation group $G$ on $\cP$ with a certain level of homogeneity or transitivity, 
and a suitable base block $B$ is the key to success for obtaining $t$-designs with computable parameters 
$t$, $v$, $k$ and $\lambda$.  If the permutation group $G$ or the base block $B$ is not properly selected, computing the parameter $k=|B|$ may be infeasible, let alone the parameter $\lambda$. 

In this paper, we considered the point set $\gf(q)$ and the permutation group $\GA_1(q)$ together with 
some base blocks, which are defined by cyclotomic classes or quadratic forms. There are a few families 
of permutation groups on $\gf(q)$ with a certain level of transitivity or homogeneity \cite[Chapter V]{BJL}. 
In principle, we may consider the $2$-designs obtained from the action of the general affine group $\GA_n(q)$ 
on the same base blocks, but it would be hard to determine the parameters of the corresponding $2$-designs for 
large $n$. 
We restricted ourselves to the action of the group $\GA_1(q)$, as this group has a very small size and 
is simple. Our 
base blocks were carefully selected, so that their block sizes and their stabilisers in $\GA_1(q)$ could 
be determined.

It in interesting to note that the group action of $\GA_1(q)$ can produce $3$-designs sometimes. 
While some people may have the opinion that $2$-designs are not very interesting due to their 
small strength, the discussions in Section \ref{sec-codesdesigns} show that $2$-designs could 
be very attractive in coding theory. It would be nice if the three conjectures presented in this 
paper could be settled. The reader is cordially invited to attack these problems.   

\section*{Acknowledgements} 

The authors are grateful to Vladimir Tonchev for helpful discussions on interactions 
between designs and codes. The first author's research was supported by the Hong Kong 
Research Grants  Council, Proj. No. 16300415.


\begin{thebibliography}{99}

\bibitem{AK92} E. F. Assmus Jr., and J. D. Key, Designs and Their Codes, Cambridge University 
Press, Cambridge, 1992.   

\bibitem{BJL} T. Beth, D. Jungnickel, H. Lenz, Design Theory, Cambridge University 
Press, Cambridge, 1999.    

\bibitem{Cameron} 
P. J. Cameron, H. R. Maimani, G. R. Omidi, B. Tayfeh-Rezaie, 
$3$-designs from $\PSL(2,q)$, Discrete Mathematics 306 (2006) 3063--3073.  

\bibitem{Carlitz} 
L. Carlitz, Explicit evaluation of certain exponential sums, 
Math. Scand. 44 (1979) 5--16.  

\bibitem{DG70} 
P. Delsarte, J.-M. Goethals, Irreducible binary cyclic codes of even dimension, 
Combinatorial Mathematics and Its Applications, Proc. Second Chapel Hill Conference, 
May 1970, Univ. of North Carolina, Chapel Hill, NC (1970), 100--113.   

\bibitem{DingLi171} 
C. Ding, C. Li, Infinite families of 2-designs and 3-designs from linear codes, 
Discrete Mathematics 340 (2017) 2415--2431.  

\bibitem{DingZhou171} 
C. Ding, Z. Zhou, Parameters of 2-designs from some BCH codes, 
in: S. El Hajji et al. (Eds.): C2SI 2017, LNCS 10194, pp. 110--127, 2017. 

\bibitem{Iwasaki2} 
S. Iwasaki, An elementary and unified approach to the Mathieu-Witt systems, 
J. Math. Soc. Japan 40(2) (1988) 393--414. 

\bibitem{Iwasaki} 
S. Iwasaki, T. Meixner, A remark on the action of $\PGL(2, q)$ and $\PSL(2,q)$ on the projective 
line, Hokkaido Mathematical Journal 26 (1997), 203--209. 

\bibitem{Klapper} 
A. Klapper, Cross-correlations of geometric sequences in characteristic two, 
Designs, Codes and Cryptography 3 (1993) 347--377.   

\bibitem{LN97} 
R. Lidl and H. Niederreiter, Finite Fields, Cambridge University Press, Cambridge, 1997. 

\bibitem{Liuetal} 
W. J. Liu, J. X. Tang, Y. X. Wu, Some new $3$-designs from $\PSL(2, q)$ with 
$q \equiv 1 \pmod{4}$, Science China Mathematics 55(9) (2012) 1901--1911.   
 

\bibitem{Storer67} 
T. Storer, \emph{Cyclotomy and Difference Sets}, Markham, Chicago, 1967. 

\bibitem{Tonc93} V. D. Tonchev, Quasi-symmetric designs, codes, quadrics, and hyperplane sections,  
Geometriae Dedicata 48 (1993)  295--308.  


\bibitem{Tonchev}  V. D. Tonchev, Codes and designs, In:   
Handbook of Coding Theory, Vol. II, V. S. Pless, and W. C. Huffman, (Editors), Elsevier, Amsterdam, 1998, pp. 1229--1268.   

\bibitem{Tonchevhb} V. D. Tonchev, Codes, In:   
Handbook of Combinatorial Designs, 2nd Edition, C. J. Colbourn, and J. H. Dinitz, (Editors), CRC Press, New York, 2007, pp. 677--701.    


\end{thebibliography}
\end{document}